\documentclass[11pt]{article}

\usepackage{latexsym}
\usepackage{amssymb}

\newtheorem{theorem}{Theorem}
\newtheorem{proposition}{Proposition}
\newtheorem{corollary}{Corollary}

\newtheorem{lemma}{Lemma}

\begin{document}
{

\begin{center}
{\Large\bf
On the orthogonality of solutions for higher-order non-Hermitian difference equations}
\end{center}

\begin{center}
{\bf Sergey M. Zagorodnyuk}
\end{center}

\noindent
\textbf{Abstract.}
In this paper we study higher-order difference equations which can be written as follows:
$$ \mathbf{J} (y_0,y_1,...)^T = \lambda^N (y_0,y_1,...)^T, $$
where $\mathbf{J}$ is a $(2N+1)$-diagonal bounded banded matrix 
($\mathbf{J}=(g_{m,n})_{m,n=0}^\infty$, $| g_{m,n} |< C$, $C>0$; and $g_{k,l}=0$ if $|k-l|>N$), 
$y_j$s are unknowns,
$\lambda$ is a complex parameter, $N\in\mathbb{N}$. It is assumed that all $g_{k,k+N}$ and $g_{l-N,l}$ are nonzero.
Two special cases are considered:

\noindent
\textit{Case A}: The matrix $\mathbf{J}$ is complex symmetric, i.e. $\mathbf{J} = \mathbf{J}^T$.

\noindent
\textit{Case B}: The matrix $\mathbf{J}$ is such that $g_{k,k+N}=1$, $k=0,1,2,...$. 
Notice that this condition can be attained by changing $y_j$s by their multiples.

In both cases there exists a \textit{positive} matrix measure $M$ on a circle in the complex plane such that polynomial solutions
satisfy some orthogonality relations. Namely, in case~A this is related to
a $J$-orthogonality in the Hilbert space $L^2(M)$ ($J$ is a complex conjugation).
In case~B we have a left $J$-orthogonality in $L^2(M)$. As a tool, a related matrix moment problem is studied.
A complex rank-one perturbation of a free Jacobi matrix is discussed.



\section{Introduction.}

The theory of real Jacobi matrices and the theory of block Hernitian Jacobi matrices 
are now  classical subjects with large applications in mathematics, physics and other 
sciences~\cite{cit_2000_Akhiezer_book_1965},\cite{cit_7000_Berezanskii}. 
One of the central results in these theories, Favard's theorem, provides orthogonality relations for
polynomial solutions of the corresponding difference equations. 
This provides an important link to the theory of (scalar and matrix) orthogonal polynomials on the real line.
On the other hand, complex Jacobi matrices or J-matrices appeared much later in Wall's book~\cite{cit_2500_Wall}.
The latter theory developed much more slowly. Maybe, the reason was that the orthogonality of polynomial solutions was given with respect to
a linear functional. This kind of orthogonality leads to the so called formal orthogonal polynomials (FOPs). 
A nice overview on this theory was done by Beckermann in 2001~\cite{cit_2100_Beckermann_2001}.
Complex analogs of block Hermitian Jacobi matrices, called (complex) banded matrices, became popular just in this century.
Formal orthogonality relations for complex symmetric $(2N+1)$-diagonal matrices were given in~\cite{cit_2500_Z_Serdica}.
Formal left-orthogonality relations in a more general case were studied 
in~\cite{cit_7100_Branquinho__Marcellan__Mendes},\cite{cit_2500_Serdica_2011}.

Our main aim here is to turn the formal orthogonality with respect to a linear functional to a $J$-orthogonality
of vectors constructed from polynomials
in a Hilbert space $L^2(M)$ of square summable vector-valued functions with respect to a positive $(N\times N)$ matrix measure $M$.
This measure will be supported on a circle in the complex plane, centered at the origin.

Let $J$ be an operator of the complex conjugation $Jf=\overline{f}$ of a vector-valued function $f$ in $L^2(M)$, with a scalar product 
denoted by $(\cdot,\cdot)$. 
Two vectors $f,g\in H$ are said to be \textit{$J$-orthogonal} if $(f,Jg) = 0$.
Observe that in the matrix case $J$ is not necessarily a conjugation as the latter term appears for operators in a Hilbert space.
Namely, the property $(Jf,Jg) = (g,f)$ is not guaranteed, when $N\geq 2$.

Let $N$ be a fixed positive integer.
Let $\mathbf{J}$ be a $(2N+1)$-diagonal bounded banded matrix. Namely, we suppose that
\begin{equation}
\label{f1_5}
 \mathbf{J} = ( g_{m,n} )_{m,n=0}^\infty,\qquad g_{m,n}\in\mathbb{C},
\end{equation}
\begin{equation}
\label{f1_7}
| g_{m,n} |< C,\qquad C>0, 
\end{equation}
\begin{equation}
\label{f1_9}
g_{k,l}=0, \mbox{ if }|k-l|>N.
\end{equation}
Moreover, we assume that
\begin{equation}
\label{f1_11}
g_{k,k+N}\not= 0,\quad  g_{l-N,l}\not=0,\qquad k\in\mathbb{Z}_+,\ l=N,N+1,....
\end{equation}

We shall study the following higher-order difference equation:
\begin{equation}
\label{f1_15}
\mathbf{J} (y_0,y_1,...)^T = \lambda^N (y_0,y_1,...)^T, 
\end{equation}
where $y_j$s are unknowns, $\lambda$ is a complex parameter.
Two special cases will be considered:

\noindent
\textit{Case A}: The matrix $\mathbf{J}$ is complex symmetric, i.e. $\mathbf{J} = \mathbf{J}^T$.

\noindent
\textit{Case B}: The matrix $\mathbf{J}$ is such that $g_{k,k+N}=1$, $k=0,1,2,...$.

Observe that case~B can be attained by changing $y_j$s by their multiples.

Let us briefly describe the content of the present paper. In Section~2 we shall prove the existence of the above-mentioned positive matrix
measure $M$. Subsection~2.1 is devoted to a special matrix moment problem in the complex plane. We provide sufficient
conditions of the solvability both in the full and in the truncated cases.
In subsection~2.2 case~A is considered and $J$-orthogonality relations in the $L^2(M)$ space are obtained.
In subsection~2.3 we study case~B and we prove left $J$-orthogonality relations in the $L^2(M)$ space.
In Section~3 we discuss a complex rank-one perturbation of a free Jacobi matrix.
We shall construct a complex and a nonnegative orthogonality weights for this case. 
Notice that Jacobi matrices of such types were studied in~\cite{cit_2100_Arlinskii_2006} 
and their relations with model dissipative operators were investigated. Here we shall see a connection with
some normal operators.

\noindent
{\bf Notations. }
As usual, we denote by $\mathbb{R}, \mathbb{C}, \mathbb{N}, \mathbb{Z}, \mathbb{Z}_+$
the sets of real numbers, complex numbers, positive integers, integers and non-negative integers,
respectively. We denote $\mathbb{T} = \{ z\in\mathbb{C}\ \arrowvert \ |z|=1 \}$,
$\mathbb{D}_r = \{ z\in\mathbb{C}\ \arrowvert \ |z|\leq r \}$, for $r>0$.
By $\mathbb{Z}_{k,l}$ we mean all integers $r$, which satisfy the following inequality:
$k\leq r\leq l$.
By $\mathbb{P}$ we mean a set of all complex polynomials.
By $\mathfrak{B}(M)$ we denote the set of all Borel subsets of a set $M\subseteq\mathbb{C}$.

Let $m,n\in\mathbb{N}$.
By $\mathbb{C}_{m\times n}$ we denote a set of all complex matrices of size $(m\times n)$.
By $\mathbb{C}_{n\times n}^\geq$ we mean a subset of $\mathbb{C}_{n\times n}$ consisting of
Hermitian non-negative matrices.
If $A\in \mathbb{C}_{m\times n}$, $A = (a_{k,l})_{1\leq k\leq m,\ 1\leq l\leq n}$, then $A^T$ means the transpose of $A$, and
$\overline{A} = ( \overline{ a_{k,l} } )_{1\leq k\leq m,\ 1\leq l\leq n}$, $A^* = (\overline{A})^T$.
We denote $E_n := (\delta_{k,l})_{1\leq k\leq n,\ 1\leq l\leq n}$, where $\delta_{k,l}$ is Kronecker's delta.

Let $N\in\mathbb{N}$.
For a $\mathbb{C}_{N\times N}^\geq$-valued measure $M(\delta)$ on $\mathfrak{B}(\mathbb{C})$ we denote by $L^2_M$ the set of all
(classes of equivalence of) $\mathfrak{B}(\mathbb{C})$-measurable 
$\mathbb{C}_{1\times N}$-valued functions $f(z) = (f_0(z),...,f_{N-1}(z))$, $z\in\mathbb{C}$, such that
$\int_\mathbb{C} f(z) M'_\tau f^*(z) d\tau < +\infty$ (see~\cite{cit_2510___Rosenberg}). 
Here $M'_\tau$ means the Radon-Nikodym derivative of $M$ with respect to the trace
measure $\tau = \tau_M$.
The class of the equivalence containing a function $f$ will be denoted by $[f]$.

By $l^2$ we denote the usual space of square-summable complex row vectors $\vec u = (u_k)_{k=0}^\infty$, $u_k\in\mathbb{C}$,
and $l^2_{\mathrm{fin}}$ means the subset of all finitely supported vectors from $l^2$. Moreover $\vec e_k$ means a vector
from $l^2$ having $1$ in $k$'s place and zeros in other places ($k\in\mathbb{Z}_+$).

If $H$ is a Hilbert space then $(\cdot,\cdot)_H$ and $\| \cdot \|_H$ mean
the scalar product and the norm in $H$, respectively. 
Indices may be omitted in obvious cases.
For a linear operator $A$ in $H$, we denote by $D(A)$
its  domain, by $R(A)$ its range, and $A^*$ means the adjoint operator
if it exists. If $A$ is invertible then $A^{-1}$ means its
inverse. $\overline{A}$ means the closure of the operator, if the
operator is closable. If $A$ is bounded then $\| A \|$ denotes its
norm.
By $E_H$ we denote the identity operator in $H$, i.e. $E_H x = x$,
$x\in H$. In obvious cases we may omit the index $H$. If $H_1$ is a subspace of $H$, then $P_{H_1} =
P_{H_1}^{H}$ denotes the orthogonal projection of $H$ onto $H_1$.
By $A|_{H_1}$ we mean the restriction of $A$ to the subspace $H_1$.

By $F(a,b;c;z)$ we denote the hypergeometric function:
$$ F(a,b;c;z) = \sum_{n=0}^\infty \frac{ (a)_n (b)_n }{ (c)_n n!} z^n,\qquad |z|<1, $$
with complex parameters $a,b,c$; $c\not= 0,-1,-2,...$. By $(a)_n$ one means Pochhammer's symbol (or the shifted factorial).

\section{The existence of a positive matrix measure for the orthogonality relations}

\subsection{A matrix moment problem in the complex plane}

In this subsection we shall study the following moment problem:
find a $\mathbb{C}_{N\times N}^\geq$-valued measure $M(\delta)$, $\delta\in\mathfrak{B}(\mathbb{C})$, satisfying the following conditions:
\begin{equation}
\label{f2_5}
\int_\mathbb{C} z^k dM(z) = S_k,\qquad k=0,1,2,...,
\end{equation}
where $S_k\in \mathbb{C}_{N\times N}$ are prescribed matrices, $N\in\mathbb{N}$. The matrices $S_k$ are called \textit{(matrix) moments}.

A scalar version of this moment problem ($N=1$) was studied in~\cite{cit_2700_Surveys_2023}.
The following theorem provides sufficient conditions for the solvability of the moment problem~(\ref{f2_5}).

\begin{theorem}
\label{t2_1} 
 Let the moment problem~(\ref{f2_5}) be given with $S_0 = E_N$, $S_k =
 (s_{k;l,j})_{0\leq l,j\leq N-1}
 \in \mathbb{C}_{N\times N}$, $k\in\mathbb{N}$.
Suppose that for some positive numbers $C$, $R$, the following relation holds:
\begin{equation}
\label{f2_7}
| s_{k;l,j} | \leq C R^k,\qquad l,j\in\mathbb{Z}_{0,N-1},\ k\in\mathbb{N}.
\end{equation}
Then the moment problem~(\ref{f2_5}) has a solution $M(\delta)$, supported on a circle centered at the origin. 
Moreover, the following representation holds:
\begin{equation}
\label{f2_8}
\oint_{|z|=1} (rz)^k W(z) d\lambda = S_k,\qquad k=0,1,2,...,
\end{equation}
where $r>0$, $\lambda$ is the Lebesgue measure on $\mathfrak{B}(\mathbb{T})$,
and $W(z)$ is a $\mathbb{C}_{N\times N}^\geq$-valued $\mathfrak{B}(\mathbb{T})$-measurable bounded function
(i.e. each entry of $W(z)$ is $\mathfrak{B}(\mathbb{T})$-measurable and bounded).

\end{theorem}

\noindent
\textbf{Proof.} For the first assertion of the theorem we shall adapt to the matrix case 
the idea of a proof of Theorem~2 in~\cite{cit_2700_Surveys_2023}.
We shall construct a contractive operator in $l^2$, and its unitary dilation will provide us with a solution to the moment problem.

At first, we suppose that the following condition holds:
\begin{equation}
\label{f2_9}
\sum_{k=1}^\infty \| S_k \|^2 < \infty,
\end{equation}
where $\| S_k \| = \left(
\sum_{l,j=0}^{N-1} | s_{k;l,j} |^2
\right)^{1/2}$.
Denote $H:= l^2$, and define the following sequence of vectors $\{ g_n \}_{n=0}^\infty$ in $H$:
\begin{equation}
\label{f2_11}
g_j := \vec e_j,\qquad j\in\mathbb{Z}_{0,N-1},
\end{equation}
\begin{equation}
\label{f2_12}
g_{kN+j} := \vec e_{kN+j} + \sum_{l=0}^{N-1} s_{k;l,j} \vec e_l,\qquad j\in\mathbb{Z}_{0,N-1},\ k\in\mathbb{N}.
\end{equation}
Observe that vectors $\{ g_n \}_{n=0}^\infty$ are linearly independent. Therefore we may define the following operator $A$
with the domain $D(A) = l^2_{\mathrm{fin}}$:
\begin{equation}
\label{f2_15}
A u = \sum_{j=0}^\infty \alpha_j g_{j+N},\quad u = \sum_{j=0}^\infty \alpha_j g_j\in l^2_{\mathrm{fin}},\ \alpha_j\in\mathbb{C},
\end{equation}
where all but finite number of $\alpha_j$s are zeros. The latter condition will be assumed in what follows, when dealing with
elements of $l^2_{\mathrm{fin}}$.
Notice that $A$ is linear and
\begin{equation}
\label{f2_17}
(A^k g_j, g_l)_H = (g_{kN+j}, \vec e_l)_H = s_{k;l,j},\qquad l,j\in\mathbb{Z}_{0,N-1},\ k\in\mathbb{N}.
\end{equation}
Let us calculate the matrix $\mathcal{M}_A$ of $A$ with respect to the standard basis $\{ \vec e_n \}_{n=0}^\infty$:
$$ \mathcal{M}_A = (m_{A;n,d})_{n,d=0}^\infty,\quad  m_{A;n,d} := (A \vec e_d, \vec e_n)_H. $$
If $d\in\mathbb{Z}_{0,N-1}$, then
$$ m_{A;n,d} = (A g_d, \vec e_n)_H = (g_{N+d}, \vec e_n)_H  
= \left(
\vec e_{N+d} + \sum_{l=0}^{N-1} s_{1;l,d} \vec e_l, \vec e_n
\right)_H = $$
\begin{equation}
\label{f2_19}
= \left\{
\begin{array}{ccc}
 s_{1;n,d}, & n\in \mathbb{Z}_{0,N-1} \\
\delta_{n,N+d}, & n\in \mathbb{Z}_{N,2N-1} \\
0, & n\geq 2N\end{array}
\right..
\end{equation}
If $d\geq N$, $d=kN+j$, $k\in\mathbb{N}$, $j\in\mathbb{Z}_{0,N-1}$, then
$$ m_{A;n,d} = \left(
A \left(
g_{kN+j} - \sum_{l=0}^{N-1} s_{k;l,j} \vec e_l
\right), \vec e_n
\right)_H = $$
$$ = \left(
g_{(k+1)N+j} - \sum_{l=0}^{N-1} s_{k;l,j} g_{N+l}, \vec e_n
\right)_H = $$
$$ = 
\left(
\vec e_{(k+1)N+j} +
\sum_{l=0}^{N-1} s_{k+1;l,j} \vec e_l -
\sum_{l=0}^{N-1} s_{k;l,j}
\left(
\vec e_{N+l} + \sum_{t=0}^{N-1} s_{1;t,l} \vec e_t
\right), \vec e_n
\right)_H =
$$
\begin{equation}
\label{f2_25}
= \left\{
\begin{array}{ccc}
s_{k+1;n,j} - \sum_{l=0}^{N-1} s_{k;l,j} s_{1;n,l}, & n\in \mathbb{Z}_{0,N-1} \\
- s_{k;n-N,j}, & n\in \mathbb{Z}_{N,2N-1} \\
\delta_{n,N+d}, & n\geq 2N\end{array}
\right..
\end{equation}
Consequently, the matrix $\mathcal{M}_A$ has the following block structure:
\begin{equation}
\label{f2_27}
\mathcal{M}_A =
\left(
\begin{array}{ccccc}
S_1 & S_2 - S_1^2 & S_3 - S_1 S_2 & S_4 - S_1 S_3 & \ldots \\
E_N & - S_1 & -S_2 & -S_3 & \ldots \\
0_{\infty\times N} &  & E &  \end{array}
\right),
\end{equation}
where $0_{\infty\times N}$ is a null matrix with the infinite number of rows and $N$ columns,
and $E$ denotes a semi-infinite matrix with units on the main diagonal and zeros elsewhere.

Denote by $\vec v_n = (v_{n;j})_{j=0}^\infty$ the $n$-th row of $\mathcal{M}_A$, $n=0,1,...,2N-1$. 
Vectors
$\vec v_N, \vec v_{N+1}, ...,\vec v_{2N-1}$ belong to $l^2$, since the partial sums for $\| \vec v_j \|^2$ ($j\in\mathbb{Z}_{N,2N-1}$)
are majorized by partial sums of the series from~(\ref{f2_9}), incremented by $1$.
Since 
$$ \left(
\begin{array}{ccccc}
S_1 & S_2 - S_1^2 & S_3 - S_1 S_2 & S_4 - S_1 S_3 & \ldots \end{array}
\right) = $$
$$ = \left(
\begin{array}{ccccc}
S_1 & S_2 & S_3 & S_4 & \ldots \end{array}  
\right) -
S_1 
\left(
\begin{array}{ccccc}
0 & S_1 & S_2 & S_3 & \ldots \end{array}
\right), $$
then $\vec v_0, \vec v_1, ...,\vec v_{N-1}$ belong to $l^2$ as well.

Consider the following shift operator, defined on the whole $l^2$:
\begin{equation}
\label{f2_29}
S \vec u = u_N \vec e_{2N} + u_{N+1} \vec e_{2N+1} + \ldots,\qquad  \vec u = (u_j)_{j=0}^\infty\in l^2.
\end{equation}
The operator $S$ is linear and bounded.
Define the following operator on the whole $l^2$:
\begin{equation}
\label{f2_31}
B \vec u = S \vec u + \sum_{m=0}^{2N-1} (\vec u, \vec w_m)_H \vec e_m,\qquad  \vec u\in l^2,
\end{equation}
where $\vec w_m := ( \overline{v_{m;j}} )_{j=0}^\infty$.
The operator $B$ is linear and bounded. Moreover, the matrix of $B$
with respect to the standard basis $\{ \vec e_n \}_{n=0}^\infty$ coincides with $\mathcal{M}_A$. 
Therefore $A$ is bounded and $B$ is its extension by the continuity to the whole $l^2$.

Choose an arbitrary positive number $\rho$ greater than  $\| B \|$. Denote
$T := \frac{1}{\rho} B$. Notice that the operator $T$ is a \textit{completely non-unitary (c.n.u.)}
contraction on $H$. We may apply the power dilation theorem
(see, e.g., \cite{cit_995_Sz.-Nagy_Book}).
Namely, there exists a unitary operator $U$ in a Hilbert space $\widetilde H\supseteq H$ such that
\begin{equation}
\label{f2_34}
P^{\widetilde H}_H U^k |_H = T^k,\qquad k\in\mathbb{Z}_+.
\end{equation}
We may write
$$ (U^k g_j, g_l)_{\widetilde H} = (T^k g_j, g_l)_H = \frac{1}{\rho^k} (B^k g_j, g_l)_H = 
\frac{1}{\rho^k} (A^k g_j, g_l)_H = $$
$$ = \frac{1}{\rho^k} (g_{kN+j}, g_l)_H = \frac{1}{\rho^k} s_{k;l,j},\qquad l,j\in\mathbb{Z}_{0,N-1},\ k\in\mathbb{N}. $$
where the last equality follows from~(\ref{f2_17}).
Then
\begin{equation}
\label{f2_36}
\left(
(\rho U)^k g_j, g_l
\right)_{\widetilde H} = s_{k;l,j},\qquad l,j\in\mathbb{Z}_{0,N-1},\ k\in\mathbb{Z}_+.
\end{equation}
Applying the functional calculus for the normal operator $\rho U$ we get
$$ \int_{|z|=\rho} z^k d(F g_j, g_l)_{\widetilde H} = s_{k;l,j},,\qquad l,j\in\mathbb{Z}_{0,N-1},\ k\in\mathbb{Z}_+, $$
where $F(\delta)$ ($\delta\in\mathfrak{B}(\mathbb{C})$) is the orthogonal spectral measure of $\rho U$.
Then 
\begin{equation}
\label{f2_38}
M(\delta) := \left( (F(\delta) g_j, g_l)_{\widetilde H} \right)_{l,j=0}^{N-1},\qquad \delta\in\mathfrak{B}(\mathbb{C}),
\end{equation}
is a required solution of the moment problem~(\ref{f2_5}).

Let us consider the general case. 
Choose an arbitrary number $\xi$ greater than $R$ from~(\ref{f2_7}).
Set
\begin{equation}
\label{f2_42}
\widehat S_k := (\widehat s_{k;l,j})_{l,j=0}^{N-1},\qquad  k\in\mathbb{Z}_+,
\end{equation}
with
\begin{equation}
\label{f2_44}
\widehat s_{k;l,j} := s_{k;l,j} / \xi^k,\qquad l,j\in\mathbb{Z}_{0,N-1},\ k\in\mathbb{Z}_+.
\end{equation}
Then
\begin{equation}
\label{f2_46}
| \widehat s_{k;l,j} | = |s_{k;l,j}| / \xi^k\leq C \left(
\frac{R}{\xi}
\right)^k,\qquad l,j\in\mathbb{Z}_{0,N-1},\ k\in\mathbb{Z}_+.
\end{equation}
Consequently, we have
$$ \sum_{k=1}^\infty \| \widehat S_k \|^2 < \infty. $$
By the first part of the proof, there exists $\widehat\rho > 0$, and a unitary operator $\widehat U$ in a Hilbert space $\widehat H\supseteq H$, 
which is a dilation of a c.n.u. contraction on $H$, such that
\begin{equation}
\label{f2_48}
\left(
(\widehat\rho \widehat U)^k g_j, g_l
\right)_{\widehat H} = \widehat s_{k;l,j},\qquad l,j\in\mathbb{Z}_{0,N-1},\ k\in\mathbb{Z}_+,
\end{equation}
where $g_js$ are now constructed for the new moments. 
Then
\begin{equation}
\label{f2_50}
\left(
(\xi \widehat\rho \widehat U)^k g_j, g_l
\right)_{\widehat H} = s_{k;l,j},\qquad l,j\in\mathbb{Z}_{0,N-1},\ k\in\mathbb{Z}_+,
\end{equation}
and we may use the spectral measure of $\xi \widehat\rho \widehat U$ to get a solution of the moment problem~(\ref{f2_5}).

Let us check the validity of representation~(\ref{f2_8}). For this purpose we shall use formula~(\ref{f2_50}).
Let $E(\delta)$ ($\delta\in\mathfrak{B}(\mathbb{T})$) be the orthogonal spectral measure of a unitary operator $\widehat U$.
Since $\widehat U$ is a dilation of a c.n.u. contraction, then all scalar measures of the following form:
$$ (E(\delta) h,h),\qquad h\in H,\ h\not= 0, $$
are equivalent to the Lebesgue measure on $\mathfrak{B}(\mathbb{T})$, see~\cite[p.~88]{cit_995_Sz.-Nagy_Book}.
Define
\begin{equation}
\label{f2_51}
\widehat M(\delta) := \left( (E(\delta) g_j, g_l)_{\widehat H} \right)_{l,j=0}^{N-1},\qquad \delta\in\mathfrak{B}(\mathbb{T}).
\end{equation}
The trace measure $\tau = \tau_{\widehat M}$ is absolutely continuous with respect to the Lebesgue measure $\lambda$
on $\mathfrak{B}(\mathbb{T})$.
Set $r = \xi \widehat\rho$.
By relation~(\ref{f2_50}) we may write:
$$  S_k = 
\oint_{|z|=1} (rz)^k d \widehat M(\delta)  =
\oint_{|z|=1} (rz)^k \widehat M'_\tau d\tau =
\oint_{|z|=1} (rz)^k \widehat M'_\tau \tau'_\lambda d\lambda,\qquad k\in\mathbb{Z}_+. $$
Set $W(z) := \widehat M'_\tau \tau'_\lambda$ to obtain relation~(\ref{f2_8}).
$\Box$

Consider the following moment problem:
find a $\mathbb{C}_{N\times N}^\geq$-valued measure $M(\delta)$, $\delta\in\mathfrak{B}(\mathbb{C})$, such that
\begin{equation}
\label{f2_55}
\int_\mathbb{C} z^k dM(z) = S_k,\qquad k\in\mathbb{Z}_{0,d},
\end{equation}
where $S_k\in \mathbb{C}_{N\times N}$ are prescribed matrices, $N,d\in\mathbb{N}$.
A scalar version of this truncated moment problem appeared in~\cite{cit_2500_Axioms_2022}.

\begin{corollary}
\label{c2_1}
 Let the moment problem~(\ref{f2_55}) be given with $S_0 = E_N$, $S_k \in \mathbb{C}_{N\times N}$, $k\in\mathbb{Z}_{1,d}$.
Then the moment problem~(\ref{f2_55}) has a solution $M(\delta)$, supported on a circle centered at the origin. 
\end{corollary}

\noindent
\textbf{Proof.}
We may define $S_k = 0$, for $k>d$. Then we apply Theorem~\ref{t2_1} to obtain a solution of the truncated moment problem.
$\Box$

\subsection{The case of complex symmetric banded matrices}

In this subsection we shall study case~A, as it was described in the Introduction. We shall need some objects from~\cite{cit_2500_Z_Serdica},
which are related to a complex symmetric banded matrix $\mathbf{J}$.

Let $\{ p_n(\lambda) \}_{n=0}^\infty$ be a sequence of polynomials which is a solution of difference equation~(\ref{f1_15}), with
the following initial conditions:

\begin{equation}
\label{f2_57}
p_j(\lambda) = \lambda^j,\qquad j\in\mathbb{Z}_{0,N-1}.
\end{equation}
We define a linear with respect to the both arguments functional $\sigma$ by the following relation and the linearity:
\begin{equation}
\label{f2_59}
\sigma(p_n(\lambda),p_m(\lambda)) = \delta_{n,m},\qquad n,m\in\mathbb{Z}_+.
\end{equation}
The functional $\sigma$ is said to be \textit{the spectral function of difference equation~(\ref{f1_15})}.
The functional $\sigma$ is \textit{symmetric}, the latter means that the following relation holds:
\begin{equation}
\label{f2_61}
\sigma(u(\lambda),v(\lambda)) = \sigma(v(\lambda),u(\lambda)),\qquad u,v\in\mathbb{P}.
\end{equation}
Moreover, the following necessary and sufficient conditions hold (cf.~\cite[Theorem 2]{cit_2500_Z_Serdica}). 

\begin{theorem}
\label{t2_2} 
A linear with respect to the both arguments symmetric functional $\sigma(u,v)$, $u,v\in\mathbb{P}$, is the spectral function
of a difference equation of type~(\ref{f1_15}) iff:
\begin{itemize}
\item[1)] $\sigma(\lambda^N u(\lambda),v(\lambda)) = \sigma(u(\lambda),\lambda^N v(\lambda)),\qquad u,v\in\mathbb{P}$;

\item[2)] $\sigma(\lambda^i,\lambda^j) = \delta_{i,j},\qquad i,j=0,1,\ldots, N-1$;

\item[3)] For arbitrary polynomial $u_k(\lambda)$ of degree $k$ there exists a polynomial $\widehat u_k(\lambda)$ of degree $k$
such that:
$$ \sigma(u_k(\lambda),\widehat u_k(\lambda)) \not= 0. $$
\end{itemize}
\end{theorem}

We shall now define
\begin{equation}
\label{f2_63}
\Gamma := (\gamma_{j,n})_{j,n=0}^\infty,\quad
\gamma_{j,n} := \sigma(\lambda^j,\lambda^n),\ j,n\in\mathbb{Z}_+. 
\end{equation}
Let us represent the matrix $\Gamma$ as a block matrix in the following way:
\begin{equation}
\label{f2_65}
\Gamma = (G_{k,l})_{k,l=0}^\infty,\quad G_{k,l} := (\gamma_{kN+r,lN+s})_{r,s=0}^{N-1}. 
\end{equation}
By property~1) of Theorem~\ref{t2_2} we may write:
$$ \gamma_{kN+r,lN+s} = \sigma(\lambda^{kN+r},\lambda^{lN+s}) = \sigma(\lambda^{(k+l)N+r},\lambda^{s}) = 
\gamma_{(k+l)N+r,s}, $$
\begin{equation}
\label{f2_66}
k,l\in\mathbb{Z}_+,\ r,s\in\mathbb{Z}_{0,N-1}.
\end{equation}
Therefore
\begin{equation}
\label{f2_67}
G_{k,l} = G_{k+l,0},\qquad k,l\in\mathbb{Z}_+,
\end{equation}
and $\Gamma$ is a block Hankel matrix. Define
\begin{equation}
\label{f2_69}
S_n = G_{n,0},\qquad n\in\mathbb{Z}_+.
\end{equation}
Notice that $S_0 = E_N$. In order to apply Theorem~\ref{t2_1} we should check the estimate~(\ref{f2_7}).
For this purpose we shall consider the following transformation~$\kappa$. 
If $u(\lambda)\in\mathbb{P}$, then it has a unique representation as a linear combination of $p_r(\lambda)$:
\begin{equation}
\label{f2_71}
u(\lambda) = \sum_{r=0}^\infty c_r p_r(\lambda),\qquad c_r\in\mathbb{C},
\end{equation}
where all but a finite number of $c_r$s are zero. We set
\begin{equation}
\label{f2_73}
\kappa u := (c_0,c_1,c_2,\ldots)^T,
\end{equation}
where the superscript $T$ means transposition.
Let us check that
\begin{equation}
\label{f2_75}
\kappa (\lambda^N u(\lambda)) = \mathbf{J} (\kappa u(\lambda)),\qquad u\in\mathbb{P}.
\end{equation}
If $u(\lambda)$ has form~(\ref{f2_71}), then using the difference equation~(\ref{f1_15}) we may write:
$$ \lambda^N u(\lambda) = \sum_{r=0}^\infty c_r \lambda^N p_r(\lambda) = 
\sum_{r=0}^\infty \sum_{j=r-N}^{r+N} c_r g_{r,j} p_j(\lambda) = 
\sum_{j=0}^\infty \sum_{r=j-N}^{j+N} c_r g_{r,j} p_j(\lambda),  $$
where $g_{r,j}$, $c_r$ and $p_j$, having negative indices, are assumed to be zero. 
If
$\kappa (\lambda^N u(\lambda)) = (\widehat c_0,\widehat c_1,\widehat c_2,\ldots)^T$, then
$$ \widehat c_j = \sum_{r=j-N}^{j+N} c_r g_{r,j} = \sum_{r=j-N}^{j+N} g_{j,r} c_r, $$
and relation~(\ref{f2_75}) follows. Here we have used the complex symmetry of $\mathbf{J}$.

The induction argument shows that
\begin{equation}
\label{f2_77}
\kappa (\lambda^{kN} p_l(\lambda)) = \mathbf{J}^k (\kappa p_l(\lambda)),\qquad l\in\mathbb{Z}_{0,N-1},\ k\in\mathbb{Z}_+.
\end{equation}
Now we can check the estimate~(\ref{f2_7}). By~(\ref{f2_69}),(\ref{f2_65}),(\ref{f2_63}) we may write
$$ |s_{k;l,j}| = |\gamma_{kN+l,j}| = | \sigma(\lambda^{kN+l},\lambda^j) |
= | \sigma(\lambda^{kN} p_l(\lambda),p_j(\lambda)) |, $$
where $l,j\in\mathbb{Z}_{0,N-1}$, $k\in\mathbb{Z}_+$.
Observe that
\begin{equation}
\label{f2_78}
\sigma(u(\lambda),v(\lambda)) = (\kappa v)^T \kappa u,\qquad u,v\in\mathbb{P}. 
\end{equation}
Then
$$ |s_{k;l,j}| = 
\left|
(\kappa p_j(\lambda))^T \kappa(\lambda^{kN} p_l(\lambda))
\right|
= 
\left|
\vec e_j \mathbf{J}^k (\vec e_l)^T
\right| = $$
$$ = \left|
(\mathcal{J}^k \vec e_l, \vec e_j)_H
\right| \leq \| \mathcal{J} \|^k, $$
where by $\mathcal{J}$ we denote the extension by the continuity on $H=l^2$ of the operator of multiplication by $\mathbf{J}$ on $l^2_{\mathrm{fin}}$.
Thus, estimate~(\ref{f2_7}) holds true and we may apply Theorem~\ref{t2_1}. 
We obtain that there exists a $\mathbb{C}_{N\times N}^\geq$-valued measure $M(\delta)$, $\delta\in\mathfrak{B}(\mathbb{C})$, such that:
\begin{equation}
\label{f2_79}
\oint_{|z|=\alpha} z^k dM(z) = G_{k,0},\qquad k\in\mathbb{Z}_+,\ \alpha>0.
\end{equation}
Consider the following operators $R_{N,m}(p)$, $m\in\mathbb{Z}_{0,N-1}$, with the domain $\mathbb{P}$:
\begin{equation}
\label{f2_81}
R_{N,m}(p)(t) = \sum_{n=0}^\infty \frac{ p^{(nN+m)}(0) }{ (nN+m)! } t^n,\qquad p\in\mathbb{P}.
\end{equation}
Operators $R_{N,m}(p)$ were intensively used by Dur\'an, see, e.g.~\cite{cit_7100_Duran__1995}.
Observe that 
\begin{equation}
\label{f2_83}
R_{N,m}(\lambda^{kN+r})(t) = \left\{
\begin{array}{cc} 0, & \mbox{if $r\not= m$} \\
t^k, & \mbox{if $r = m$} \end{array}
\right.,\qquad r\in\mathbb{Z}_{0,N-1},\ k\in\mathbb{Z}_+.
\end{equation}
By~(\ref{f2_65}),(\ref{f2_66}),(\ref{f2_79}) we may write:
$$ \gamma_{kN+r,lN+s} = \gamma_{(k+l)N+r,s} = \vec e_r S_{k+l} (\vec e_s)^* = 
\oint_{|z|=\alpha} z^k \vec e_r M'_\tau (\vec e_s)^* z^l d\tau = $$
$$ = \oint_{|z|=\alpha} 
\left(
R_{N,0}(z^{kN+r}), R_{N,1}(z^{kN+r}),\ldots, R_{N,N-1}(z^{kN+r})
\right)
M'_\tau * $$
\begin{equation}
\label{f2_84}
* \left(
\begin{array}{cccc} R_{N,0}(z^{lN+s}) \\
R_{N,1}(z^{lN+s}) \\ 
\ldots \\
R_{N,N-1}(z^{lN+s})\end{array}
\right) d\tau,\qquad k,l\in\mathbb{Z}_+,\ r,s\in\mathbb{Z}_{0,N-1}, 
\end{equation}
where $\tau = \tau_M$ is the trace measure for $M$.
Therefore
$$ \sigma(z^n,z^j) = \gamma_{n,j} = 
\oint_{|z|=\alpha} 
\left(
R_{N,0}(z^n), R_{N,1}(z^n),\ldots, R_{N,N-1}(z^n)
\right)
M'_\tau * $$
\begin{equation}
\label{f2_85}
* \left(
\begin{array}{cccc} R_{N,0}(z^j) \\
R_{N,1}(z^j) \\ 
\ldots \\
R_{N,N-1}(z^j)\end{array}
\right) d\tau,\qquad n,j\in\mathbb{Z}_+.
\end{equation}
By the linearity of $\sigma$ in both arguments we obtain that
$$ \sigma(u,v) = 
\oint_{|z|=\alpha} 
\left(
R_{N,0}(u)(z), R_{N,1}(u)(z),\ldots, R_{N,N-1}(u)(z)
\right)
M'_\tau * $$
\begin{equation}
\label{f2_87}
* \left(
\begin{array}{cccc} R_{N,0}(v)(z) \\
R_{N,1}(v)(z) \\ 
\ldots \\
R_{N,N-1}(v)(z)\end{array}
\right) d\tau,\qquad u,v\in\mathbb{P}.
\end{equation}
By the definition of the spectral function $\sigma$ we obtain the following orthogonality relations:
$$ \oint_{|z|=\alpha} 
\left(
R_{N,0}(p_n)(z), R_{N,1}(p_n)(z),\ldots, R_{N,N-1}(p_n)(z)
\right)
M'_\tau * $$
\begin{equation}
\label{f2_89}
* \left(
\begin{array}{cccc} R_{N,0}(p_m)(z) \\
R_{N,1}(p_m)(z) \\ 
\ldots \\
R_{N,N-1}(p_m)(z)\end{array}
\right) d\tau = \delta_{n,m},\qquad n,m\in\mathbb{Z}_+.
\end{equation}

\begin{theorem}
\label{t2_3}
Let $N\in\mathbb{N}$ and $\mathbf{J}$ be a semi-infinite complex symmetric matrix of form~(\ref{f1_5}) satisfying
conditions~(\ref{f1_7})-(\ref{f1_11}). 
Let $\{ p_n(\lambda) \}_{n=0}^\infty$ be a sequence of polynomials which is a solution of difference equation~(\ref{f1_15}), with
the initial conditions~(\ref{f2_57}). 
Then there exists a $\mathbb{C}_{N\times N}^\geq$-valued measure $M(\delta)$, $\delta\in\mathfrak{B}(\mathbb{C})$, 
supported on a circle $\{ z\in\mathbb{C}\ \arrowvert \ |z| = \alpha \}$, $\alpha>0$, such that
relation~(\ref{f2_89}) holds.
\end{theorem}

\noindent
\textbf{Proof.} The proof follows from considerations before the statement of the theorem. $\Box$

\subsection{The case of not necessarily symmetric banded matrices}

This subsection is devoted to a discussion on banded matrices of case~B (see the Introduction). 
We shall use some objects from~\cite{cit_2500_Serdica_2011}, which are related to a banded matrix $\mathbf{J}$ within case~B.

Let $\{ p_n(\lambda) \}_{n=0}^\infty$ be a sequence of \textit{monic} polynomials which is a solution of difference equation~(\ref{f1_15}), with
the following initial conditions:

\begin{equation}
\label{f2_91}
p_j(\lambda) = \lambda^j,\qquad j\in\mathbb{Z}_{0,N-1}.
\end{equation}

In case~B the spectral function is defined in a quite another way, as comparing with case~A.
A sesquilinear functional $\sigma=\sigma(u,v)$, $u,v\in\mathbb{P}$ (i.e. linear in the first argument, 
antilinear in the second argument, but not necessarily $\sigma(u,v)=\overline{\sigma(v,u)}$) which
satisfies the following relations:
\begin{equation}
\label{f2_94}
\sigma(p_n(\lambda),p_j(\lambda)) = \delta_{n,j},\qquad n,j\in\mathbb{Z}_{0,N-1};
\end{equation}
\begin{equation}
\label{f2_96}
\sigma(p_n(\lambda),p_j(\lambda)) = 0,\qquad n,j\in\mathbb{Z}_+:\ n>j;
\end{equation}
\begin{equation}
\label{f2_98}
\sigma(\lambda^N u(\lambda),v(\lambda)) = \sigma(u(\lambda), \lambda^N v(\lambda)),\qquad u,v\in\mathbb{P},
\end{equation}
is said to be \textit{the (generalized) spectral function of the matrix $\mathbf{J}$}.
In~\cite{cit_2500_Serdica_2011} it was shown that the spectral function of $\mathbf{J}$ is unique, and
a procedure of the construction of $\sigma$ was provided.
The following necessary and sufficient conditions were proved in~\cite[Theorem 1]{cit_2500_Serdica_2011}.

\begin{theorem}
\label{t2_4} 
A sesquilinear functional $\sigma(u,v)$, $u,v\in\mathbb{P}$, is the spectral function
of a complex banded matrix $\mathbf{J} = ( g_{m,n} )_{m,n=0}^\infty$, $g_{m,n}\in\mathbb{C}$, 
satisfying relations~(\ref{f1_9}),(\ref{f1_11}) and 
such that $g_{k,k+N}=1$, $k\in\mathbb{Z}$,
if and only if:
\begin{itemize}
\item[1)] $\sigma(\lambda^N u(\lambda),v(\lambda)) = \sigma(u(\lambda),\lambda^N v(\lambda)),\qquad u,v\in\mathbb{P}$;

\item[2)] $\sigma(\lambda^k,\lambda^l) = \delta_{k,l},\qquad k,l=0,1,\ldots, N-1$;

\item[3)] $\det\Gamma_M\not= 0$, $M\in\mathbb{Z}_+$, where $\Gamma_M = (\gamma_{k,l})_{k,l=0}^N$, 
$\gamma_{k,l} = \sigma(\lambda^k,\lambda^l)$.
\end{itemize}
\end{theorem}

Suppose that case~$B$ holds for a banded matrix $\mathbf{J} = ( g_{m,n} )_{m,n=0}^\infty$, $g_{m,n}\in\mathbb{C}$. 
Let $\sigma$ be the (generalized) spectral function of $\mathbf{J}$.
Set
\begin{equation}
\label{f2_99}
\gamma_{j,n} := \sigma(\lambda^j,\lambda^n),\qquad j,n\in\mathbb{Z}_+.
\end{equation}
The following matrix $\Gamma$:
\begin{equation}
\label{f2_100}
\Gamma := (\gamma_{j,n})_{j,n=0}^\infty,
\end{equation}
is a $(N\times N)$ block Hankel matrix:
\begin{equation}
\label{f2_104}
\Gamma = (G_{k+l})_{k,l=0}^\infty,\qquad G_j\in\mathbb{C}_{N\times N},
\end{equation}
and $G_0 = E_N$, see~\cite[Theorem 2]{cit_2500_Serdica_2011}.
In what follows we can repeat many of the arguments applied in the case~A.
As before, in order to apply Theorem~\ref{t2_1} we should check the estimate~(\ref{f2_7}).
For this purpose we shall use the transformation~$\kappa$ from~(\ref{f2_73}). 
Since  in case~B the matrix $\mathbf{J}$ is not necessarily complex symmetric, relation~(\ref{f2_75}) is replaced
by the following relation:
\begin{equation}
\label{f2_106}
\kappa (\lambda^N u(\lambda)) = \mathbf{J}^T (\kappa u(\lambda)),\qquad u\in\mathbb{P}.
\end{equation}
Using the induction argument we obtain that
\begin{equation}
\label{f2_108}
\kappa (\lambda^{kN} p_l(\lambda)) = (\mathbf{J}^T)^k (\kappa p_l(\lambda)),\qquad l\in\mathbb{Z}_{0,N-1},\ k\in\mathbb{Z}_+.
\end{equation}
Relation~(\ref{f2_78}) remains valid.
Let 
$$ S_k := G_k,\quad   S_k = (s_{k;l,j})_{l,j=0}^{N-1},\ s_{k;l,j}\in\mathbb{C},\quad k\in\mathbb{Z}_+. $$
We may write:
$$ |s_{k;l,j}| =
|\gamma_{kN+l,j}| = | \sigma(\lambda^{kN+l},\lambda^j) |
= | \sigma(\lambda^{kN} p_l(\lambda),p_j(\lambda)) | = $$
$$ = \left|
(\kappa p_j)^T \kappa(\lambda^{kN} p_l(\lambda))
\right|
= 
\left|
\vec e_j (\mathbf{J}^T)^k (\vec e_l)^T
\right| = 
\left|
(\mathcal{C}^k \vec e_l, \vec e_j)_H
\right| \leq \| \mathcal{C} \|^k, $$
where by $\mathcal{C}$ we denote the extension by the continuity on $H=l^2$ of the operator of multiplication by $\mathbf{J}^T$ on $l^2_{\mathrm{fin}}$.
So, estimate~(\ref{f2_7}) is valid. 
By Theorem~\ref{t2_1} there exists a $\mathbb{C}_{N\times N}^\geq$-valued measure $M(\delta)$, $\delta\in\mathfrak{B}(\mathbb{C})$, such that:
\begin{equation}
\label{f2_110}
\oint_{|z|=\alpha} z^k dM(z) = G_k,\qquad k\in\mathbb{Z}_+,\ \alpha>0.
\end{equation}
Relations~(\ref{f2_84}),(\ref{f2_85}) are valid now. However, in case~B the spectral function~$\sigma$ is not linear
with respect to the second argument. 
We define a new functional~$\widehat\sigma$ in the following way:
$$ \widehat\sigma(u,v) = 
\oint_{|z|=\alpha} 
\left(
R_{N,0}(u)(z), R_{N,1}(u)(z),\ldots, R_{N,N-1}(u)(z)
\right)
M'_\tau * $$
\begin{equation}
\label{f2_112}
* \left(
\begin{array}{cccc} R_{N,0}(v)(z) \\
R_{N,1}(v)(z) \\ 
\ldots \\
R_{N,N-1}(v)(z)\end{array}
\right) d\tau,\qquad u,v\in\mathbb{P}.
\end{equation}
This functional is linear with respect to the both arguments and it coincides with the spectral function $\sigma$ on monomials.
Thus, the functional $\widehat\sigma$ can be recovered from the spectral function $\sigma$ and vice versa.
We shall call $\widehat\sigma$ \textit{the additional spectral function of the matrix $\mathbf{J}$}.

By the linearity in the first argument we obtain that
\begin{equation}
\label{f2_114}
\widehat\sigma(u(\lambda),\lambda^j) = \sigma(u(\lambda),\lambda^j),\qquad u\in\mathbb{P},\ j\in\mathbb{Z}_+.
\end{equation}
Let
$$ p_j(\lambda) = \sum_{r=0}^j \alpha_{j,r} \lambda^r,\qquad \alpha_{j,r}\in\mathbb{C},\ j\in\mathbb{Z}_+. $$
By~(\ref{f2_114}),(\ref{f2_96}) we may write:
$$ \widehat\sigma(p_n(\lambda),p_j(\lambda)) = \sum_{r=0}^j \alpha_{j,r} \widehat\sigma(p_n(\lambda), \lambda^r) = 
\sum_{r=0}^j \alpha_{j,r} \sigma(p_n(\lambda), \lambda^r) = $$
\begin{equation}
\label{f2_116}
= \sigma\left( p_n(\lambda), \sum_{r=0}^j \overline{\alpha_{j,r}} \lambda^r \right) = 0,\qquad n,j\in\mathbb{Z}_+:\ n>j,
\end{equation}
since $\sum_{r=0}^j \overline{\alpha_{j,r}} \lambda^r$ is a linear combination of $p_k$s with $k<n$.

\begin{theorem}
\label{t2_5}
Let $N\in\mathbb{N}$ and $\mathbf{J}$ be a semi-infinite matrix of form~(\ref{f1_5}) satisfying
conditions~(\ref{f1_7})-(\ref{f1_11}), and $g_{k,k+N}=1$, $k=0,1,2,...$. 
Let $\{ p_n(\lambda) \}_{n=0}^\infty$ be a sequence of polynomials which is a solution of difference equation~(\ref{f1_15}), with
the initial conditions~(\ref{f2_91}). 
Then there exists a $\mathbb{C}_{N\times N}^\geq$-valued measure $M(\delta)$, $\delta\in\mathfrak{B}(\mathbb{C})$, 
supported on a circle $\{ z\in\mathbb{C}\ \arrowvert \ |z| = \alpha \}$, $\alpha>0$, such that
$$ \oint_{|z|=\alpha} 
\left(
R_{N,0}(p_n)(z), R_{N,1}(p_n)(z),\ldots, R_{N,N-1}(p_n)(z)
\right)
M'_\tau * $$
\begin{equation}
\label{f2_120}
* \left(
\begin{array}{cccc} R_{N,0}(p_m)(z) \\
R_{N,1}(p_m)(z) \\ 
\ldots \\
R_{N,N-1}(p_m)(z)\end{array}
\right) d\tau = \eta_n \delta_{n,m},\qquad \eta_n\not= 0,\ n,m\in\mathbb{Z}_+:\ n\geq m.
\end{equation}
\end{theorem}

\noindent
\textbf{Proof.} The proof of formula~(\ref{f2_120}) for $n>m$ follows from considerations before the statement of the theorem. 
Suppose that for some $n\in\mathbb{Z}_+$ we have $\widehat\sigma(p_n,p_n)=0$.
Representing $p_n$ as a linear combination of $\lambda^n$ and $p_r$s with $r<n$, and using relations~(\ref{f2_116}),(\ref{f2_114})
we obtain that
$$ 0 = \widehat\sigma(p_n,\lambda^n) = \sigma(p_n,\lambda^n). $$
The last relation contradicts to relation~(24) in~\cite{cit_2500_Serdica_2011} (in the proof of Theorem~1 in~\cite{cit_2500_Serdica_2011}
it was shown that $\sigma(p_n(\lambda),\lambda^n))\not = 0$ for a spectral function $\sigma$).
$\Box$

\section{A complex rank-one perturbation of a free Jacobi matrix}

Let $N=1$ and $\mathbf{J} = \mathbf{J}_c$ be a semi-infinite matrix of form~(\ref{f1_5}) with
$g_{0,0}=c\in\mathbb{C}$, $g_{k,k+1}=1$, $k\in\mathbb{Z}_+$, $g_{j-1,j}=1$, $j\in\mathbb{N}$,
and the rest of the elements of $\mathbf{J}$ being zeros.
Of course, if the parameter $c$ is real we obtain a real Jacobi matrix.
Let $\{ p_n(\lambda) \}_{n=0}^\infty$ be a sequence of polynomials which is a solution of difference equation~(\ref{f1_15}), with
the initial condition $p_0(\lambda)=1$.
We define a linear with respect to the both arguments functional $\sigma$ by relations~(\ref{f2_59}).
The functional $\sigma$ is the spectral function of the corresponding difference equation~(\ref{f1_15}).
Notice that the latter difference equation in our case can be written as
\begin{equation}
\label{f3_3}
c y_0 + y_1 = \lambda y_0,
\end{equation}
\begin{equation}
\label{f3_4}
y_{n-1} + y_{n+1} = \lambda y_n,\qquad n\in\mathbb{N}.
\end{equation}
We shall need Chebyshev's polynomials of the second kind $U_n(x)$ and their scaled version $u_n(x)$:
\begin{equation}
\label{f3_5}
U_n(x) := \frac{ \sin((n+1)\arccos x) }{ \sqrt{1-x^2} },\quad u_n(x) := U_n(x/2),\quad n\in\mathbb{Z}_+,\ x\in(-1,1).
\end{equation}
Polynomials $p_n(\lambda)$ have the following explicit representation:
\begin{equation}
\label{f3_7}
p_n(\lambda) = u_n(\lambda) - c u_{n-1}(\lambda),\qquad n\in\mathbb{Z}_+,\ u_{-1}(\lambda) := 0..
\end{equation}
In fact, the right-hand side expression satisfies relations~(\ref{f3_3}),(\ref{f3_4}) and 
the initial conditions are the same.
On the other hand, one can express $u_n$ in terms of $p_n$ in the following way:
\begin{equation}
\label{f3_9}
u_n(\lambda) = \sum_{j=0}^n c^{n-j} p_j(\lambda),\qquad n\in\mathbb{Z}_+.
\end{equation}
This can be checked by the induction argument using relation~(\ref{f3_7}). For $n=0$ relation~(\ref{f3_9}) is valid. Suppose that it holds
for some $k\in\mathbb{Z}_+$. By~(\ref{f3_7}) we may write:
$$ u_{k+1}(\lambda) = p_{k+1}(\lambda) + c u_k(\lambda) = p_{k+1}(\lambda) + c \sum_{j=0}^k c^{k-j} p_j(\lambda) = 
\sum_{j=0}^{k+1} c^{k+1-j} p_j(\lambda). $$

For an arbitrary positive integer $k$ the following relations hold:
\begin{equation}
\label{f3_11}
z^{2k} = (2k)! \sum_{j=0}^k \frac{ (2j+1) }{ (k-j)! (k+j+1)! } u_{2j}(z),
\end{equation}
\begin{equation}
\label{f3_14}
z^{2k-1} = (2k-1)! \sum_{j=1}^k \frac{ 2j }{ (k-j)! (k+j)! } u_{2j-1}(z).
\end{equation}
These relations follow from more general relations for the Gegenbauer polynomials, 
see formula~(36) on page~283 in~\cite{cit_5150_Rainville}. 
Using expressions for $u_n$ from~(\ref{f3_9}) we obtain that
\begin{equation}
\label{f3_15}
z^{2k} = (2k)! \sum_{j=0}^k \sum_{m=0}^{2j} \frac{ (2j+1) }{ (k-j)! (k+j+1)! } c^{2j-m} p_m(z),
\end{equation}
\begin{equation}
\label{f3_17}
z^{2k-1} = (2k-1)! \sum_{j=1}^k \sum_{m=0}^{2j-1} \frac{ 2j }{ (k-j)! (k+j)! } c^{2j-1-m} p_m(z),\qquad k\in\mathbb{N}.
\end{equation}
By property~1) of Theorem~\ref{t2_2} it follows that
$$ \sigma(u,v) = \sigma(uv,1),\qquad u,v\in\mathbb{P}. $$
We may define the following functional:
\begin{equation}
\label{f3_19}
S(u) = \sigma(u,1),\qquad u\in\mathbb{P}.
\end{equation}
Then
\begin{equation}
\label{f3_22}
S(p_n p_m) = \delta_{n,m},\qquad n,m\in\mathbb{Z}_+.
\end{equation}
Denote
\begin{equation}
\label{f3_24}
s_n := S(\lambda^n),\qquad n\in\mathbb{Z}_+.
\end{equation}
Applying $S$ to the both sides of relations~(\ref{f3_15}),(\ref{f3_17}) and using orthogonality we get
\begin{equation}
\label{f3_26}
s_{2k} = (2k)! \sum_{j=0}^k \frac{ (2j+1) }{ (k-j)! (k+j+1)! } c^{2j},\qquad k\in\mathbb{Z}_+.
\end{equation}
\begin{equation}
\label{f3_28}
s_{2k-1} = (2k-1)! \sum_{j=1}^k \frac{ 2j }{ (k-j)! (k+j)! } c^{2j-1},\qquad k\in\mathbb{N}.
\end{equation}

At first, we shall present a complex-valued weight of orthogonality for~$p_n$.
Denote by $[x]$ the biggest integer, not exceeding a real number $x$.

\begin{proposition}
\label{p3_1}
Let $R_0 = \max(2, 2|c|)$. 
For each $R>R_0$ the following relations hold:
\begin{equation}
\label{f3_30}
\frac{1}{2\pi i} \oint_{|z|=R} p_n(z) p_m(z) w(z) dz = \delta_{n,m},\qquad n,m\in\mathbb{Z}_+,
\end{equation}
where
\begin{equation}
\label{f3_32}
w(z) := \sum_{n=1}^\infty n c^{n-1} z^{-n} \varphi_n(z),
\end{equation}
\begin{equation}
\label{f3_34}
\varphi_n(z) := \sum_{m=0}^\infty \frac{ (2m+n-1)! }{ m! (m+n)! } z^{-2m},\qquad z\in\mathbb{C}:\ |z| > R_0.
\end{equation}
The series in~(\ref{f3_32}),(\ref{f3_34}) converge absolutely and uniformly for $|z|\geq R$.
The function $w(z)$ is analytic in the domain $|z|>R_0$.
Moreover, $\varphi_n(z)$ has the following representation:
\begin{equation}
\label{f3_36}
\varphi_n(z) = a_n
F \left( \frac{n}{2}, \frac{n+1}{2}; n+1; \frac{4}{z^2}
\right),\qquad n\in\mathbb{N},
\end{equation}
where
\begin{equation}
\label{f3_38}
a_n = \frac{ 2^{n-1} }{n!} \left[ \frac{n-1}{2} \right] ! \left( \frac{1}{2} \right)_{ \left[ \frac{n}{2} \right] }.
\end{equation}

\end{proposition}

\noindent
\textbf{Proof.} Let us define $\varphi_n$ by relation~(\ref{f3_36}) and then check the validity of~(\ref{f3_34}).
Since the hypergeometric function converges in the open unit circle, then $\varphi_n$
is well-defined for for $|z|>2$.
Observe that
$$ a_{2j} = \frac{2^{2j-1}}{(2j)!} (j-1)! (1/2)_j,\quad
a_{2j-1} = \frac{2^{2j-2}}{(2j-1)!} (j-1)! (1/2)_{j-1},\qquad j\in\mathbb{N}. $$
By the definition of the hypergeometric function we may write:
$$ \varphi_{2j}(z) = \frac{2^{2j-1}}{(2j)!} (j-1)! (1/2)_j
\sum_{m=0}^\infty \frac{ (j)_m \left( j+\frac{1}{2} \right)_m }{ (2j+1)_m m! } 4^m z^{-2m} = $$
$$ = \frac{1}{2} \sum_{m=0}^\infty \frac{ (j+m-1)! \left( \frac{1}{2} \right)_{j+m} 4^{j+m} }{ (2j+m)! m! } z^{-2m} = 
\sum_{m=0}^\infty \frac{ (2j+2m-1)! }{ (2j+m)! m! } z^{-2m},\ j\in\mathbb{N}.  $$
This agrees with relation~(\ref{f3_34}) for the case $n=2j$. The case $n=2j-1$, $j\in\mathbb{N}$, can be checked in a similar manner:
$$ \varphi_{2j-1}(z) = \frac{2^{2j-2}}{(2j-1)!} (j-1)! (1/2)_{j-1}
\sum_{m=0}^\infty \frac{ \left( j-\frac{1}{2} \right)_m (j)_m }{ (2j)_m m! } 4^m z^{-2m} = $$
$$ = \sum_{m=0}^\infty \frac{ (j+m-1)! \left( \frac{1}{2} \right)_{j+m-1} 4^{j+m-1} }{ (2j+m-1)! m! } z^{-2m} = 
\sum_{m=0}^\infty \frac{ (2j+2m-2)! }{ (2j+m-1)! m! } z^{-2m}, $$
and this agrees with~(\ref{f3_34}) as well.
By the convergence property of hypergeometric series it follows that $\varphi_n$ converges
absolutely and uniformly for $|z|\geq R$.
If $n=2j$, $j\in\mathbb{N}$, $m\in\mathbb{N}$, then
$$ \frac{ (2m+n-1)! }{ m! (m+n)! } = \frac{ (2m+2j-1)! }{ m! (m+2j)! }
= $$
$$ = \frac{ 1\cdot 3\cdot 5\cdot ... \cdot (2(m+j)-1) }{m!} \cdot
\frac{ 2\cdot 4\cdot 6\cdot ...\cdot 2(m+j-1) }{ (m+2j)!} \leq $$
$$ \leq \frac{ 2^{m+j} (m+j)! }{m!} \cdot \frac{ 2^{m+j-1} (m+j-1)! }{ (m+2j)! } \leq $$
$$ \leq 2^{2m+2j-1} \frac{1}{m+j} \frac{ (m+1)(m+2)...(m+j) }{ (m+1+j)(m+2+j)...(m+2j) } \leq $$ 
\begin{equation}
\label{f3_42}
\leq 2^{2m+2j-1} \frac{1}{m+j} = 2^{n-1} 4^m \frac{2}{2m + 2j} \leq 2^{n} 4^m \frac{1}{n}. 
\end{equation}
If $n=2j+1$, $j\in\mathbb{Z}_+$, $m\in\mathbb{N}$, then
$$ \frac{ (2m+n-1)! }{ m! (m+n)! } = \frac{ (2m+2j)! }{ m! (m+2j+1)! } = $$
$$ = \frac{ 2\cdot 4\cdot 6\cdot ... \cdot 2(m+j) }{m!} \cdot
\frac{ 1\cdot 3\cdot 5\cdot ...\cdot (2(m+j)-1) }{ (m+2j+1)!} \leq $$
$$ \leq \frac{ 2^{m+j} (m+j)! }{m!} \cdot
\frac{ 2^{m+j} (m+j)! }{ (m+2j+1)! }. $$
If $j\geq 1$ then the last expression is equal to
$$ 4^{m+j} \frac{1}{m+2j+1} \frac{ (m+1)(m+2)...(m+j) }{ (m+1+j)(m+2+j)...(m+2j) } \leq $$ 
\begin{equation}
\label{f3_44}
\leq 4^{m+j} \frac{1}{m+2j+1} \leq 4^{m+j} \frac{1}{n} \leq 2^{n} 4^m \frac{1}{n}. 
\end{equation}
If $j=0$ then the above-mentioned expression is equal to $\frac{1}{2} 4^m$, and therefore it is less than or equal to
$2^{n} 4^m \frac{1}{n}$, as well. Finally, if $m=0$ and $n\in\mathbb{N}$, then
$\frac{ (2m+n-1)! }{ m! (m+n)! } =  \frac{1}{n} \leq 2^{n} 4^m \frac{1}{n}$.
We may write
$$ | \varphi_n(z) | \leq  \frac{ 2^{n} }{n}
\sum_{m=0}^\infty 4^m \frac{1}{ |z|^{2m} } = \frac{ 2^{n} }{n} \frac{ |z|^2 }{ |z|^2 - 4 },\qquad  |z|>R_0,\ n\in\mathbb{N}. $$ 
For arbitrary $n\in\mathbb{N}$ we may write:
$$ | n c^{n-1} z^{-n} \varphi_n(z) | \leq 
|c|^{n-1}
\frac{1}{ |z|^n } 2^n \frac{ |z|^2 }{ (|z|^2 - 4) }
\leq 
\frac{ |z|^2 }{ |c| (|z|^2 - 4) } 
\left(
\frac{ 2 |c| }{ |z| }
\right)^n \leq $$
$$ \leq
\frac{ |z|^2 }{ |c| (|z|^2 - 4) } 
(R_0/R)^n \leq C_1 (R_0/R)^n,\qquad   C_1>0,\ |z|\geq R. $$
Thus, $w(z)$ converges absolutely and uniformly for $|z|\geq R$.
By the Weierstrass theorem, $w(z)$ is an analytic function in the domain $|z|>R_0$.

Let us define the following functional:
\begin{equation}
\label{f3_44_a}
\widehat S(u) := \frac{1}{2\pi i} \oint_{|z|=R} u(z) w(z) dz,\qquad u\in\mathbb{P}. 
\end{equation}
Our aim is to check that $\widehat S = S$.
For an arbitrary $d\in\mathbb{Z}$ we may write:
$$ \frac{1}{2\pi i} \oint_{|z|=R} z^d w(z) dz =
\frac{1}{2\pi i} \oint_{|z|=R}  
\sum_{n=1}^\infty n c^{n-1} z^{d-n} \varphi_n(z) dz = $$
$$ = \sum_{n=1}^\infty n c^{n-1} \frac{1}{2\pi i} \oint_{|z|=R} z^{d-n} \varphi_n(z) dz = $$
$$ = \sum_{n=1}^\infty n c^{n-1} \frac{1}{2\pi i} \oint_{|z|=R}  
\sum_{m=0}^\infty \frac{ (2m+n-1)! }{ m! (m+n)! } z^{d-n-2m}
dz = $$
$$ = \sum_{n=1}^\infty \sum_{m=0}^\infty n c^{n-1} 
\frac{ (2m+n-1)! }{ m! (m+n)! }
\frac{1}{2\pi i} \oint_{|z|=R}  
z^{d-n-2m}
dz = $$
\begin{equation}
\label{f3_44_1}
= 
\sum_{n=1}^\infty \sum_{m=0}^\infty
n c^{n-1} 
\frac{ (2m+n-1)! }{ m! (m+n)! }
\delta_{2m,d+1-n}, 
\end{equation}
where we have used the uniform convergence of the series to change the order of summation and integration.
In particular, if $d=2k$. $k\in\mathbb{Z}_+$ we have:
$$ \widehat S(z^{2k}) = 
\sum_{n=1}^\infty \sum_{m=0}^\infty
n c^{n-1} 
\frac{ (2m+n-1)! }{ m! (m+n)! }
\delta_{2m,2k+1-n}. $$
Notice that there only remain non-zero summands for $n=2j+1$, $j=0,1,...,k$, and $m=k-j$. Then
$$ \widehat S(z^{2k}) = 
\sum_{j=0}^k 
(2j+1) c^{2j} 
\frac{ (2k)! }{ (k-j)! (k+j+1)! } = s_{2k}. $$
In a similar manner, for an arbitrary $k\in\mathbb{N}$ we obtain that
$\widehat S(z^{2k-1}) = s_{2k-1}$. By the linearity we get $\widehat S = S$. 
Relation~(\ref{f3_30}) now follows from relation~(\ref{f3_22}). $\Box$

We shall now construct a nonnegative weight function for the orthogonality relations of $p_n$.
By Theorem~\ref{t2_3} there exists a positive (scalar) measure $M(\delta)$, $\delta\in\mathfrak{B}(\mathbb{C})$, 
supported on a circle $\{ z\in\mathbb{C}\ \arrowvert \ |z| = \alpha \}$, $\alpha>0$, such that
the following relation holds:
\begin{equation}
\label{f3_45}
\oint_{|z|=\alpha} p_n(z) p_m(z) dM(z)
= \delta_{n,m},\qquad n,m\in\mathbb{Z}_+.
\end{equation}
Moreover, the measure $M$ was obtained as a solution of the moment problem~(\ref{f2_5}), by application of Theorem~\ref{t2_1}
(see relation~(\ref{f2_79})).
From the proof of Theorem~\ref{t2_1} it is clear that the radius $\alpha$ can be chosen arbitrarily large.
We choose $\alpha> R_0$, $R_0 = \max(2, 2|c|)$.
We shall use relation~(\ref{f3_30}) with $R=\alpha$:
\begin{equation}
\label{f3_48}
\frac{1}{2\pi i} \oint_{|z|= \alpha} p_n(z) p_m(z) w(z) dz = \delta_{n,m},\qquad n,m\in\mathbb{Z}_+.
\end{equation}
Observe that an addition to $w(z)$ of any analytic function in a circle $|z|< \beta$, with $\beta > \alpha$,
does not change the value of the latter integral.
Define the following function
\begin{equation}
\label{f3_50}
g(z) := \frac{1}{R^2} \sum_{j=0}^\infty \overline{ s_{j+1} } \left( 
\frac{z}{R^2} \right)^j,
\end{equation}
for those $z$, where the series converges.

\begin{lemma}
\label{l3_1}
The function $g(z)$ is well-defined and analytic in $\mathbb{D}_{R_1}$, 
$R_1 := \frac{ R^2 }{ R_0 }$ ($> R$).
It has the following representation:
\begin{equation}
\label{f3_52}
g(z) = \frac{1}{ R^2 } \widehat w\left( \frac{z}{R^2} \right),\qquad |z| < \frac{ R^2 }{ R_0 }, 
\end{equation}
where
\begin{equation}
\label{f3_54}
\widehat w(u) := \frac{ 1 }{ u^2 } \overline{w}\left( \frac{1}{u} \right) - \frac{1}{u},\qquad     |u|<\frac{1}{ R_0 },
\end{equation}
and $\overline{w}(\cdot) := \overline{ w(\overline{\cdot}) }$.

\end{lemma}

\noindent
\textbf{Proof.} Since $w(z)$ is analytic in $R_0 < |z| < \infty$, then it has a representation by the Laurent series in this domain:
\begin{equation}
\label{f3_56}
w(z) = \sum_{n=-\infty}^{+\infty} c_n z^n,
\end{equation}
where
\begin{equation}
\label{f3_58}
c_n = \frac{1}{2\pi i } \oint_{ |z|=\alpha } z^{-n-1} w(z) dz.
\end{equation}
By relation~(\ref{f3_44_1}) we conclude that $c_n$ with nonnegative indices are zeros. By considerations after~(\ref{f3_44_1})
we see that $c_n = s_{-n-1}$, for negative $n$. Thus, we have the following expansion of $w(z)$:
\begin{equation}
\label{f3_60}
w(z) = \sum_{n=-\infty}^{-1} s_{-n-1} z^n,\qquad R_0 < |z| < \infty.
\end{equation}
Then
\begin{equation}
\label{f3_62}
\overline{w}(z) = \sum_{n=-\infty}^{-1} \overline{ s_{-n-1} } z^n,\qquad R_0 < |z| < \infty,
\end{equation}
and
\begin{equation}
\label{f3_64}
\overline{w}(1/u) = \sum_{n=-\infty}^{-1} \overline{ s_{-n-1} } u^{-n} = 
\sum_{k=0}^\infty \overline{ s_k } u^{k+1},\qquad  |u| < \frac{1}{ R_0 }.
\end{equation}
Therefore
\begin{equation}
\label{f3_68}
\frac{ 1 }{ u^2 } \overline{w}\left( \frac{1}{u} \right) - \frac{1}{u} =
\sum_{k=1}^\infty \overline{ s_k } u^{k-1} 
= \sum_{j=0}^\infty \overline{ s_{j+1} } u^j,\qquad  |u| < \frac{1}{ R_0 }.
\end{equation}
Thus, the function $\widehat w(u)$ in~(\ref{f3_54}) is well-defined. We may write:
\begin{equation}
\label{f3_70}
\frac{1}{ R^2 } \widehat w\left( \frac{z}{R^2} \right) =
\frac{1}{R^2} \sum_{j=0}^\infty \overline{ s_{j+1} } \left( 
\frac{z}{R^2} \right)^j,\qquad |z| < \frac{ R^2 }{ R_0 }.  
\end{equation}
Comparing relation~(\ref{f3_70}) with relations (\ref{f3_50}),(\ref{f3_52})
we conclude that the function $g(z)$ is well-defined in $\mathbb{D}_{R_1}$ and
representation~(\ref{f3_52}) is valid.
$\Box$

Denote
\begin{equation}
\label{f3_72}
\mu_n = \oint_{|z|=\alpha} z^n dM(z),\quad
\tau_n = \frac{ 1 }{ 2\pi i} \oint_{|z|=\alpha} z^n (w(z) + g(z)) dz,\qquad n\in\mathbb{Z}.
\end{equation}
By (\ref{f3_45}),(\ref{f3_48}) we see that for $m\in\mathbb{Z}_+$
$$ \oint_{|z|=\alpha} p_m(z) dM(z) = \frac{ 1 }{ 2\pi i} \oint_{|z|=\alpha} p_m(z) (w(z) + g(z)) dz. $$
Since $z^n$, $n\in\mathbb{Z}_+$, can be represented as a linear combination of $p_m$s, then $\mu_n = \tau_n$, for
$n\in\mathbb{Z}_+$. For $n<0$ we may write:
$$ \overline{ \mu_n } = \oint_{|z|=\alpha} \overline{z}^n dM(z) = R^{2n} \oint_{|z|=\alpha} z^{-n} dM(z) =
R^{2n} \mu_{-n} = R^{2n} s_{-n}. $$
Therefore
\begin{equation}
\label{f3_74}
\mu_n = R^{2n} \overline{ s_{-n} },\qquad n=-1,-2,....
\end{equation}
By~(\ref{f3_44_1}),(\ref{f3_50}), for $n<0$ we have:
\begin{equation}
\label{f3_76}
\tau_n = \frac{ 1 }{ 2\pi i} \oint_{|z|=\alpha} z^n (w(z) + g(z)) dz = 
\frac{ 1 }{ 2\pi i} \oint_{|z|=\alpha} z^n g(z) dz
= R^{2n} \overline{ s_{-n} }.
\end{equation}
By~(\ref{f3_74}),(\ref{f3_76}) we conclude that $\mu_n = \tau_n$, for all $n\in\mathbb{Z}$.
Denote
\begin{equation}
\label{f3_77}
m(\theta) := M( [\alpha,\alpha e^{i\theta}) ),
\end{equation}
where by $[\alpha,\alpha e^{i\theta})$ one means an arc of the circle $|z|=\alpha$, with arguments from $0$ to $\theta$.
Then
\begin{equation}
\label{f3_78}
\int_0^{2\pi} e^{in\theta} dm(\theta) = \frac{1}{2\pi} \int_0^{2\pi} e^{in\theta} 
( w(\alpha e^{ i \theta }) + g(\alpha e^{ i \theta }) ) \alpha e^{ i \theta } d\theta,\qquad n\in\mathbb{Z}.
\end{equation}
Set
\begin{equation}
\label{f3_80}
b(\theta) := \frac{1}{2\pi} \int_0^\theta ( w(\alpha e^{ix}) + g(\alpha e^{ix}) ) \alpha e^{ix} dx.
\end{equation}
Integrating by parts we obtain that
\begin{equation}
\label{f3_82}
\int_0^{2\pi} e^{in\theta} m(\theta) d\theta = \int_0^{2\pi} e^{in\theta} b(\theta) d\theta,\qquad n\in\mathbb{Z}\backslash\{ 0 \}.
\end{equation}
Then
\begin{equation}
\label{f3_84}
m(\theta) - b(\theta) = C = \mbox{const},\qquad \mbox{a.e. on $[0,2\pi]$}.
\end{equation}
By the left continuity equality in~(\ref{f3_84}) holds for all $\theta\in (0,2\pi]$.
Therefore $b(\theta)$ is a non-decreasing function on the interval $(0,2\pi]$.
By~(\ref{f3_80}) we see that
\begin{equation}
\label{f3_86}
p(\theta) := \frac{1}{2\pi} ( w(\alpha e^{i\theta}) + g(\alpha e^{i\theta}) ) \alpha e^{i\theta} \geq 0,\qquad \theta\in (0,2\pi).
\end{equation}
By~(\ref{f3_48}) we may write:
\begin{equation}
\label{f3_88}
\int_0^{2\pi} p_n(\alpha e^{i\theta}) p_m(\alpha e^{i\theta}) p(\theta) d\theta = \delta_{n,m},\qquad n,m\in\mathbb{Z}_+.
\end{equation}

\begin{theorem}
\label{t3_1}
There exists $\alpha> R_0$, $R_0 = \max(2, 2|c|)$, such that orthogonality relations~(\ref{f3_88}) hold
with a nonnegative continuous weight $p(\theta)$ from~(\ref{f3_86}).
 
\end{theorem}

\noindent
\textbf{Proof.} The statement of the theorem follows from the preceding considerations. $\Box$

\noindent
Address:

Kharkiv, Ukraine

(In 1995-2025 worked at V.N. Karazin Kharkiv National University)

Sergey.M.Zagorodnyuk@gmail.com

}

\end{document}